\documentclass[a4paper]{article}

\usepackage{amsmath}
\usepackage{amsfonts}
\usepackage{amssymb}         
\usepackage{amsopn}
\usepackage{theorem}          
\usepackage{latexsym}
\usepackage{graphicx}        
\usepackage{mathrsfs}		
\usepackage{psfrag}
\usepackage{algorithmic}
\usepackage{algorithm}
\usepackage{color}
\usepackage{verbatim}
\usepackage{enumerate}
\usepackage{url}

\definecolor{darkgrey}{gray}{0.4}

\newtheorem{result}{\ }[section]
\theoremstyle{changebreak}                
\newtheorem{thm}[result]{Theorem}

\newtheorem{lem}[result]{Lemma}
\newtheorem{cor}[result]{Corollary}
\newtheorem{prop}[result]{Proposition}
\newtheorem{eg}[result]{Example}
\newenvironment{proof}
 {{\sl Proof.}\hspace*{1 ex}}%
 {{\nopagebreak\hspace*{\fill}$\Box$\par\vspace{12pt}}}

\newcommand{\qed}{\hfill \ensuremath{\Box}}

\begin{document}

\thispagestyle{empty}
\begin{center} 

{\LARGE Using the Johnson-Lindenstrauss lemma in linear and integer programming}
\par \bigskip
{\sc Vu Khac Ky\footnote{Supported by a Microsft Research Ph.D.~fellowship.}, Pierre-Louis Poirion, Leo Liberti} 
\par \bigskip
{\it LIX, \'Ecole Polytechnique, F-91128 Palaiseau,
France} \\ Email:\url{{vu,poirion,liberti}@lix.polytechnique.fr}
\par \medskip \today
\end{center}
\par \bigskip

\date{\today}

\begin{abstract} 
The Johnson-Lindenstrauss lemma allows dimension reduction on real
vectors with low distortion on their pairwise Euclidean
distances. This result is often used in algorithms such as $k$-means
or $k$ nearest neighbours since they only use Euclidean distances, and
has sometimes been used in optimization algorithms involving the
minimization of Euclidean distances. In this paper we introduce a
first attempt at using this lemma in the context of feasibility
problems in linear and integer programming, which cannot be expressed
only in function of Euclidean distances. 
\end{abstract}

\section{Introduction}

In machine learning theory there is a wonderful and deep result from
functional analysis and geometric convexity called the {\it
  Johnson-Lindenstrauss Lemma} (JLL) \cite{jllemma}. Intuitively, this
lemma employs a concentration of measure \cite{ledoux2005} argument to
prove that a ``cloud'' of high-dimensional points can be projected in
a much lower dimensional space whilst keeping Euclidean distances
approximately the same. Although this result was previously exploited
in purely Euclidean distance based algorithms such as $k$-means
\cite{boutsidis2010} and $k$ nearest neighbours \cite{Indyk} (among
others), it has not often been applied to optimization problems. There
are a few exceptions, namely high dimensional linear regression
\cite{pilanci2014}, where the application of the JLL is reasonably
natural.

In this paper we present some new results on the application of the
JLL to establish the feasibility of Linear Programs (LP) and Integer
Linear Programs (ILP). We consider problems with $m$ constraints
(where $m$ is large), and reduce $m$ by projection on a random
subspace, while ensuring that, with high probability, the reformulated
problem is feasible if and only if the original problem is feasible.

The geometrical intuition underlining our idea stems from the cone
interpretation of LP feasibility. Let $P$ be a feasibility-only LP in
standard form (i.e.~all inequalities have been turned into equations
by the introduction of $m$ additional non-negative variables), written
as $Ax=b$ with $x\ge 0$, where $A$ is an $m\times n$ rational matrix,
$b\in\mathbb{R}^m$ is a rational vector, and $x$ is a vector of $n$
decision variables (which might be continuous or integer). Then $P$
can be interpreted as the following geometric decision problem: given
a cone spanned by the columns of $A$, is $b\in\mbox{\sf cone}(A)$ or
not? In this setting, the role of the JLL is seen to be the following:
if we project $A$ and $b$ in a smaller dimensional space, the JLL
assures us that the ``shape" of the projected cone and of the ray $b$
are approximately the same, and hence that the answer to the problem
will be the same most of the times.


In Section 2, we formally define the problem. In Section 3, we recall
the Johnson-Lindenstrauss Lemma and prove some results linked to its
application to the ILP case. In Section 4, we derive results
concerning LP feasibility when the cone generated by the matrix $A$ is
pointed. In Section 5, we generalize the previous results, proving
that the distance between a point and a closed set should be
approximately preserved. Finally in Section 6, we present some
computational results.

\section{Linear programming and the cone membership problem}
It is well-known that any linear program can be reduced (via an easy
bisection argument) to LP feasibility, defined as follows:
\begin{quote}
{\sc Linear Feasibility Problem (LFP)}. 
Given $b \in \mathbb{R}^m$ and $A \in \mathbb{R}^{m \times n}$. Decide
whether there exists $x \in \mathbb{R}^n$ such that $Ax = b\land x \ge 0$.
\end{quote}
We assume that $m$ and $n$ are very large integer numbers. We also
assume that $A$ is full row-rank. In particular, we have $m \le n$,
since otherwise we can find $x$ uniquely from $Ax = b$ by taking the
left inverse of $A$.

LFP problems can obviously be solved using the simplex
method. Despite the fact that simplex methods are often very efficient
in practice, there are instances for which the methods run in
exponential time. On the other hand, polynomial time algorithms such
as interior point methods are known to scale poorly, in practice, on
several classes of instances.

If $a_1,\ldots,a_n$ are the column vectors of $A$, then the LFP is
equivalent to finding $x \ge 0$ such that $b$ is a non-negative linear
combination of $a_1,\ldots,a_n$. In other words, the LFP
is equivalent to the following cone membership
problem:
\begin{quote}
  {\sc Cone Membership (CM)}.  Given $b,a_1,\ldots,a_n \in
  \mathbb{R}^{m}$, decide whether $b \in \mbox{\sf cone}
  \{a_1,\ldots,a_n\}$.
\end{quote}

This problem can be viewed as a special case of the \emph{restricted
  linear membership problem}, which is defined as follows:
\begin{quote} {\sc Restricted Linear Membership
    (RLM)}. Given $b,a_1,\ldots,a_n \in \mathbb{R}^{m}$ and $X
  \subseteq \mathbb{R}^n$, decide whether $b \in \mbox{\sf
    lin}_X(a_1,\ldots,a_n)$, i.e.~whether $\exists\lambda\in X$
  s.t.~$b=\sum\limits_{i=1}^{n}\lambda_i a_i$.
\end{quote}
For example, when $X = \mathbb{R}^n_+$, we have the cone membership
problem; and when $X = \mathbb{Z}^n $ (or $\{0,1\}^n$) we have the
integer (binary) cone membership problem.

It is known from the JLL (see below for an exact statement) that there
is a (linear) mapping $T:\mathbb{R}^m\to\mathbb{R}^k$, where $k\ll m$,
such that the pairwise distances between all vector pairs $(a_i,a_j)$
undergo low distortion. In other words, the complete graph on
$\{a_1,\ldots,a_n\}$ weighted with the pairwise Euclidean distances
realized in $\mathbb{R}^m$ can also be approximately realized in
$\mathbb{R}^k$. We are now stipulating that such a graph is a
reasonable representation of the intuitive notion of ``shape''. Under
this hypothesis, it is reasonable to expect that the image of
$C=\mbox{\sf cone}(a_1,\ldots,a_n)$ under $T$ has approximately the
same shape as $C$. Thus, given an instance of CM, we expect to be able
to ``approximately solve'' a much smaller (randomly projected)
instance instead. Notice that since CM is a decision problem,
``approximately'' really refers to a randomized algorithm which is
successful with high probability.

Notationwise, every norm $\|\cdot\|$ is Euclidean unless otherwise
specified, and we shall denote by $E^{\mbox{\sf\scriptsize c}}$ the
complement of an event $E$.

\section{Random projections and RLM problems}
\label{section2}
The JLL is stated as follows:
\begin{thm}[Johnson-Lindenstrauss Lemma \cite{jllemma}] \label{J-L lemma}
Given $\varepsilon \in (0,1)$ and $A=\{a_1,\ldots,a_n\}$ be a set of
$n$ points in $\mathbb{R}^m$. Then there exists a mapping
$T:\mathbb{R}^m \to \mathbb{R}^k$ such that
\begin{equation}
  (1-\varepsilon)\|a_i -a_j\| \le \|T(a_i) - T(a_j)\| \le
  (1+\varepsilon)\|a_i -a_j\| \label{eq:jll}
\end{equation}
for all $1\le i,j\le n$, in which $k$ is $O(\varepsilon^{-2}\log n)$.
\end{thm}
Thus, all sets of $n$ points can be projected to a subspace having
dimension logarithmic in $n$ (and, surprisingly, independent of
$m$), such that no distance is distorted by more than
$1+2\varepsilon$. The JLL is a consequence of a general property (see
Lemma \ref{randomprojectionlemma} below) of \emph{sub-Gaussian} random
mappings $T(x) = \sqrt{\frac{1}{k}} P\,x$ where $P$ is an
appropriately chosen matrix. Some of the most popular are:
\begin{itemize}
\item orthogonal projections on a random $k$-dimensional linear
  subspace of $\mathbb{R}^m$;
\item random $k\times m$ matrices with each entry independently drawn
  from the standard normal distribution $\mathcal{N}(0,1)$;
\item random $k\times m$ matrices with each entry independently taking
  values $+1$ and $-1$, each with probability $\frac{1}{2}$;
\item random $k\times m$ matrices with entries independently taking
  values $+1$, $0$, $-1$, respectively with probability $\frac{1}{6}$,
  $\frac{2}{3}$, $\frac{1}{6}$.
\end{itemize}
\begin{lem}[Random projection lemma] \label{randomprojectionlemma}
Let $T:\mathbb{R}^m\to\mathbb{R}^k$ be one of the above random
mappings. Then for all $\varepsilon\in(0,1)$ and all vector
$x\in\mathbb{R}^m$, we have: 
\begin{equation} \label{eqn:01}
  \mbox{\sf Prob}(\,(1-\varepsilon)\|x\|\le\|T(x)\|\le(1+\varepsilon)\|x\|\,)
  \ge 1-2e^{-\mathcal{C}\varepsilon^2 k}
\end{equation}
for some constant $\mathcal{C}>0$ (independent of $m,k,\varepsilon$).
\end{lem}
Not only can this lemma prove the existence of a mapping satisfying
conditions in the Johnson-Lindenstrauss lemma, but also it implies
that the probability of finding such a mapping is very large. Indeed,
from the random projection lemma, the probability that
Eq.~\eqref{eq:jll} holds  for all $i\not=j\le m$ is at least
\begin{equation}
  1 - 2 {m\choose 2}e^{-\mathcal{C}\varepsilon^2 k} = 1 -
  m(m-1)e^{-\mathcal{C}\varepsilon^2 k}.\label{eq:probjll}
\end{equation}
Therefore, if we want this probability to be large than, say $99.9\%$,
then simply choose any $k$ such that $\frac{1}{100m(m-1)} >
e^{-\mathcal{C}\varepsilon^2 k}$.  This means $k$ can be chosen to be $k =
\lceil\frac{\ln(1000)+2\ln(m)}{\mathcal{C}\varepsilon^2}\rceil$, which is
$O(\,\varepsilon^{-2} (\ln(m)+3.5)\,)$.

We shall also need a squared version of the random projection lemma.
\begin{lem}[Random projection lemma, squared version]
  \label{randomprojectionlemma2}
Let $T:\mathbb{R}^m \to \mathbb{R}^k$ be one of the random mappings in
Lemma \ref{randomprojectionlemma}. Then for all $\varepsilon \in
(0,1)$ and all vector $x \in \mathbb{R}^m$, we have:
\begin{equation} \label{eqn:01a}
\mbox{\sf Prob}(\,(1-\varepsilon)\|x\|^2 \le \|T(x)\|^2 \le (1+\varepsilon)\|x\|^2\,) \ge 1 - 2e^{-\mathcal{C}(\varepsilon^2-\varepsilon^3) k}
\end{equation}
for some constant $\mathcal{C} > 0$ (independent of $m,k,\varepsilon$).
\end{lem}

Another direct consequence of the random projection lemma is the
concentration around zero of the involved random linear projection
kernel.
\begin{cor}
Let $T:\mathbb{R}^m \to \mathbb{R}^k$ be one of the random mappings as
in Lemma \ref{randomprojectionlemma} and $0 \neq x \in
\mathbb{R}^m$. Then we have 
\begin{equation}
\mbox{\sf Prob}(T(x) \neq 0) \ge 1 -
2e^{-\mathcal{C}k}. \label{eq:cor1} 
\end{equation}
for some constant $\mathcal{C}>0$ (independent of $n,k$).
\label{cor1}
\end{cor}
\begin{proof}
For any $\varepsilon\in (0,1)$ , we define the following events:
\begin{eqnarray*}
\mathcal{A} & =& \big\{T(x) \neq 0 \big \} \\
\mathcal{B} & =& \big \{(1-\varepsilon)\|x\| \le \|T(x)\| \le
 (1+\varepsilon)\|x\| \big\}.
\end{eqnarray*}
By Lemma \ref{randomprojectionlemma} it follows that $\mbox{\sf
  Prob}(\mathcal{B}) \ge 1 - 2e^{-c\varepsilon^2 k}$ for some constant
$\mathcal{C}>0$ independent of $m,k,\varepsilon$.  On the other hand,
$\mathcal{A}^{\mbox{\sf\scriptsize c}} \cap \mathcal{B} =\emptyset$,
since otherwise, there is a mapping $T_1$ such that $T_1(x) = 0$ and
$(1-\varepsilon)\|x\| \le \|T_1(x)\|,$ which altogether imply that $x
= 0$ (a contradiction). Therefore, $\mathcal{B}\subseteq\mathcal{A}$,
and we have $\mbox{\sf Prob}(\mathcal{A}) \ge \mbox{\sf
  Prob}(\mathcal{B}) \ge 1 - 2e^{-\mathcal{C}\varepsilon^2 k}$. This
holds for all $0 < \varepsilon < 1$, so $\mbox{\sf
  Prob}(\mathcal{A})\ge 1-2 e^{\mathcal{C}k}$.\qed
\end{proof}  
\begin{lem} \label{membershiplemma}
Let $T:\mathbb{R}^m \to \mathbb{R}^k$ be one of the random mappings as
in Lemma \ref{randomprojectionlemma} and
$b,a_1,\ldots,a_n\in\mathbb{R}^m$. Then for any given vector $x \in
\mathbb{R}^n$, we have:
\begin{enumerate}[(i)]
\item If $b = \sum\limits_{i=1}^n x_i a_i$ then
  $T(b)=\sum\limits_{i=1}^n x_i T(a_i)$; \label{lm1}
\item If $b \neq \sum_{i=1}^n x_i a_i$ then $\mbox{\sf Prob}\, \bigg[
  T(b) \neq \sum_{i=1}^n x_i T(a_i)\bigg] \ge 1 -
  2e^{-\mathcal{C}k}$; \label{lm2}
\item If $b \neq \sum_{i=1}^n y_i a_i$ for all $y \in X \subseteq
  \mathbb{R}^n$, where $|X|$ is finite, then
  \[\mbox{\sf Prob}\, \bigg[\forall y\in X\ 
    T(b) \neq \sum_{i=1}^n y_i T(a_i) \bigg] \ge 1 -
  2|X|e^{-\mathcal{C}k};\] \label{lm3}
\end{enumerate}
for some constant $\mathcal{C}>0$ (independent of $n,k$).
\end{lem}
\begin{proof}
 Point \eqref{lm1} follows by linearity of $T$, and \eqref{lm2} by
  applying Cor.~\ref{cor1} to $Ax-b$. For
\eqref{lm3}, we have
\begin{eqnarray*}
\mbox{\sf Prob}\, \bigg[\forall y\in X\ T(b) \neq \sum_{i=1}^n y_i
  T(a_i) \bigg] & = & \mbox{\sf Prob}\, \bigg[ \bigcap_{y \in X} \; \big
  \{ T(b) \neq \sum_{i=1}^n y_i T(a_i)\big \}\bigg] \\ = 1 -
\mbox{\sf Prob}\, \bigg[ \bigcup_{y \in X} \; \big \{ T(b) \neq
  \sum_{i=1}^n y_i T(a_i)\big \}^{\mbox{\sf\scriptsize c}}\bigg] &
  \ge& 1 - \sum_{y \in X} \mbox{\sf Prob}\, \bigg[\big \{ T(b) \neq
    \sum_{i=1}^n y_i T(a_i)\big \}^{\mbox{\sf\scriptsize c}}\bigg] \\
  \mbox{\color{darkgrey} [by \eqref{lm2}]\hspace*{1cm}} &
    \ge& 1 - \sum_{y \in X} 2e^{-\mathcal{C}k} = 1 -
    2|X|e^{-\mathcal{C}k},
\end{eqnarray*}
as claimed.\qed
\end{proof}

This lemma can be used to solve the RLM problem when the cardinality
of the restricted set $X$ is bounded by a polynomial in $n$. In
particular, if $|X| < n^d$, where $d$ is small w.r.t.~$n$, then
\begin{equation}
  \mbox{\sf Prob} \big[T(b) \notin \mbox{\sf Lin}_X \, \{T(a_1),\ldots,T(a_n)\}
  \big] \ge 1 - 2n^de^{-\mathcal{C}k}.\label{eq:rlm}
\end{equation}
Then by taking any $k$ such that $k \ge
\frac{1}{\mathcal{C}}\ln(\frac{2}{\delta}) + \frac{d}{\mathcal{C}} \ln
n$, we obtain a probability of success of at least $1 -\delta$. We
give an example to illustrate that such a bound for $|X|$ is natural
in many different settings.
\begin{eg} If $X = \{x \in \{0,1\}^n| \,\sum_{i=1}^{n} \alpha_i x_i
  \le d\}$ for some $d$, where $0 < \alpha_i$ for all $1 \le i \le n$,
  then $|X| < n^{\overline{d}}$, where $\overline{d} =
  \max\limits_{1\le i \le n} \lfloor \frac{d}{\alpha_i} \rfloor$. To
  see this, let $\underline{\alpha} = \min\limits_{1\le i \le n}
  \alpha_i$; then $\sum\limits_{i=1}^n x_ i \le \sum\limits_{i=1}^n
  \frac{\alpha_i}{\underline{\alpha}} x_ i \le
  \frac{d}{\underline{\alpha}}$, which implies $\sum\limits_{i=1}^n x_
  i \le \overline{d}$. Therefore $|X| \le {n \choose 0} + {n \choose
    1} + \ldots {n \choose \overline{d}} < n^{\overline{d}}$, as
  claimed.  \qed
\end{eg}

Lemma \ref{membershiplemma} also gives us an indication as to why
estimating the probability that
\[T(b) \notin \mbox{\sf cone}\{T(a_1),\ldots,T(a_n)\}\]
is not straightforward. This event
can be written as an intersection of infinitely many sub-events
\[\{ T(b) \neq \sum_{i=1}^n y_i T(a_i)\}\]
where $y\in\mathbb{R}_+^n$; even
if each of these occurs with high probability, their intersection
might still be small. As these events are dependent, however, we still
hope to find to find a useful estimation for this probability.

\section{Projections of separating hyperplanes}
In this section we show that if a hyperplane separates a point $x$
from a closed and convex set $C$, then its image under a random
projection $T$ is also likely to separate $T(x)$ from $T(C)$. The
separating hyperplane theorem applied to cones can be stated as
follows.
\begin{thm}[Separating hyperplane theorem]
  \label{sephyp}
  Given $b \notin \mbox{\sf cone}\{a_1,\ldots,a_n\}$ where
  $b,a_1,\ldots,a_n \in \mathbb{R}^m$. Then there is $c \in
  \mathbb{R}^m$ such that $c^Tb < 0$ and $c^Ta_i \ge 0$ for all $i =
  1,\ldots,n$.
\end{thm}

For simplicity, we will first work with \emph{pointed cone}. Recall
that a cone $C$ is called pointed if and only if $C \cap -C =
\{0\}$. The associated separating hyperplane theorem is obtained by
replacing all $\ge$ inequalities by strict ones.
Without loss of generality, we can assume that $\|c\|=1$. From this
theorem, it immediately follows that there is a positive
$\varepsilon_0$ such that $c^Tb < -\varepsilon_0$ and $c^Ta_i >
\varepsilon_0$ for all $1 \le i \le n$.

\begin{prop} Given $b, a_1,\ldots,a_n \in \mathbb{R}^{m}$ of norms $1$
  such that $b \notin \mbox{\sf cone} \{a_1,\ldots,a_n\}$, 
  $\varepsilon > 0$, and $c\in\mathbb{R}^m$ with $\|c\|=1$ be such
  that $c^Tb < -\varepsilon$ and $c^Ta_i \ge \varepsilon$ for all
  $1 \le i \le n$. Let $T:\mathbb{R}^m \to \mathbb{R}^k$ be one of the
  random mappings as in Lemma \ref{randomprojectionlemma2}, then
  \[  \mbox{\sf Prob} \big [T(b) \notin \mbox{\sf cone}
    \{T(a_1),\ldots,T(a_n)\} \big] \ge 1 -
  4(n+1)e^{-\mathcal{C}(\varepsilon^2 - \varepsilon^3)k}\] for some
  constant $\mathcal{C}$ (independent of $m,n,k,\varepsilon$).
\end{prop}
\begin{proof}
Let $A$ be the event that both $(1-\varepsilon)\|c-x\|^2 \le
\|T(c-x)\|^2 \le (1+\varepsilon)\|c-x\|^2$ and
$(1-\varepsilon)\|c+x\|^2 \le \|T(c+x)\|^2 \le
(1+\varepsilon)\|c+x\|^2$ hold for all $x \in
\{b,a_1,\ldots,a_n\}$. By Lemma \ref{randomprojectionlemma2}, we have
$\mbox{\sf Prob} (A) \ge 1 - 4(n+1)e^{-c(\varepsilon^2-\varepsilon^3)
  k}$.  For any random mapping $T$ such that $A$ occurs, we have
\begin{eqnarray*}
\langle T(c), T(b) \rangle & =& \frac{1}{4} (\|T(c+b)\|^2 - \|T(c-b\|^2) \\
& \le& \frac{1}{4} (\|c+b\|^2 - \|c-b\|^2) + \frac{\varepsilon}{4} (\|c+b\|^2 + \|c-b\|^2) \\
& =& c^Tb + \varepsilon < 0
\end{eqnarray*}
and, for all $i=1,\ldots,n$, we can similarly derive $c^Ta_i -
\varepsilon \ge 0$ from $\langle T(c), T(a_i) \rangle$. Therefore, by
Thm.~\ref{sephyp}, $T(b) \notin \mbox{\sf cone} \{T(a_1),\ldots,
T(a_n)\}$. \qed
\end{proof}

From this proposition, it follows that the larger $\varepsilon$ will
provide us a better probability. The largest $\varepsilon$ can be
found by solving the following optimization problem.
\begin{quote}
  {\sc Separating Coefficient Problem (SCP)}.\\ Given $b \notin
  \mbox{\sf cone}\,\{a_1,\ldots,a_n\}$, find $\varepsilon = \,
  \max\limits_{c, \varepsilon} \; \{\varepsilon |\, \varepsilon \ge 0,
  c^Tb \le -\varepsilon, c^Ta_i \ge \varepsilon \}$.
\end{quote}

Note that $\varepsilon$ can be extremely small when the cone $C$
generated by $a_1,\ldots,a_n$ is almost non-pointed, i.e. the convex
hull of $a_1,\ldots,a_n$ contains a point close to $0$. Indeed, for
any convex combination $x = \sum_i \lambda_i a_i$ with $\sum_i
\lambda_i = 1$ of $a_i$'s, we have:
$$\|x\| = \|x\|\,\|c\| \ge c^Tx = \sum_{i=1}^n \lambda_i c^Ta_i \ge \sum_{i=1}^n \lambda_i \varepsilon = \varepsilon.$$
Therefore, $\varepsilon \le \min \{ \|x\|\;|\; x \in \mbox{\sf conv} \{a_1,\ldots,a_n\} \}.$

\section{Projection of minimum distance}
In this section we show that if the distance between a point $x$ and a
closed set is positive, it remains positive with high probability
after applying a random projection. First, we consider the following
problem.
\begin{quote} {\sc Convex Hull Membership (CHM)}.\\
Given $b,a_1,\ldots,a_n \in \mathbb{R}^{m}$, decide whether
$b\in\mbox{\sf conv} \{a_1,\ldots,a_n\}$.
\end{quote}

\begin{prop} 
Given $a_1,\ldots,a_n \in \mathbb{R}^{m}$, let $C=\mbox{\sf conv}
\{a_1,\ldots,a_n\}$, $b \in \mathbb{R}^m$ such that $b \notin C$, $d =
\min\limits_{x \in C} \|b-x\|$ and $D = \max\limits_{1\le i \le n} \|b
-a_i\|$. Let $T:\mathbb{R}^m \to \mathbb{R}^k$ be a random mapping as
in Lemma \ref{randomprojectionlemma}. Then
\begin{equation}
  \mbox{\sf Prob} \big [T(b) \notin T(C) \big] \ge 1 - 2n^2 e^{-\mathcal{C}(\varepsilon^2-\varepsilon^3)k}
  \label{eqproj}
\end{equation}
for some constant $\mathcal{C}$ (independent of $m,n,k,d,D$) and
$\varepsilon < \frac{d^2}{D^2}$.
\end{prop}

\begin{proof}
Let $S_\varepsilon$ be the event that both $(1-\varepsilon)\|x-y\|^2
\le \|T(x-y)\|^2 \le (1+\varepsilon)\|x-y\|^2$ and
$(1-\varepsilon)\|x+y\|^2 \le \|T(x+y)\|^2 \le
(1+\varepsilon)\|x+y\|^2$ hold for all $x,y \in \{0, b-a_1,\ldots,b
-a_n\}$. Assume $S_\varepsilon$ occurs. Then for all real $\lambda_i
\ge 0$ with $\sum\limits_{i=1}^n \lambda_i = 1$, we have:
{\small
\begin{eqnarray*}
&& \|T(b)- \sum_{i=1}^n \lambda_i T(a_i) \|^2 
 =  \|\sum_{i=1}^n \lambda_i T(b - a_i) \|^2 \\ [-0.2em]
& = & \sum_{i=1}^n \lambda^2_i \|T(b - a_i) \|^2 + 2 \sum_{1\le i<j\le n} \lambda_i \lambda_j \langle T(b - a_i), T(b-a_j) \rangle \\ [-0.2em]
 &=& \sum_{i=1}^n \lambda^2_i \|T(b - a_i) \|^2 + \frac{1}{2}\sum_{1\le i<j\le n} \lambda_i \lambda_j\bigg ( \|T(b - a_i + b - a_j)\|^2 - \|T(a_i-a_j)\|^2 \bigg) \\ [-0.2em]
 &\ge & (1 - \varepsilon) \sum_{i=1}^n \lambda^2_i \|b - a_i\|^2 +  \frac{1}{2}\sum_{1\le i<j\le n} \lambda_i \lambda_j \bigg ( (1 - \varepsilon) \big \|b - a_i + b - a_j \big \|^2  - (1 + \varepsilon) \|a_i - a_j\|^2 \bigg)\\ [-0.2em]
 &= & \|b- \sum_{i=1}^n \lambda_i a_i\|^2 
- \varepsilon \bigg(\sum_{i=1}^n \lambda^2_i \|b - a_i\|^2 + \frac{1}{2}\sum_{1\le i<j\le n} \lambda_i \lambda_j ( \|b - a_i + b-a_j\|^2 + \|a_i - a_j\|^2 )\bigg) \\ [-0.2em]
 &=& \|b- \sum_{i=1}^n \lambda_i a_i\|^2 
- \varepsilon \bigg(\sum_{i=1}^n \lambda^2_i \|b - a_i\|^2 + \sum_{1\le i<j\le n} \lambda_i \lambda_j ( \|b - a_i\|^2 + \|b - a_j\|^2 )\bigg).
\end{eqnarray*}
}%
From the definitions of $d$ and $D$, we have:
{\small
\begin{eqnarray*}
  \|T(b)- \sum_{i=1}^n
\lambda_i T(a_i) \|^2 \ge d^2 - \varepsilon D^2 \bigg(\sum_{i=1}^n
  \lambda^2_i + 2 \sum_{1\le i<j\le n} \lambda_i \lambda_j \bigg) 
= d^2 - \varepsilon D^2 \bigg(\sum_{i=1}^n \lambda_i\bigg)^2 = 
\end{eqnarray*}
} $=d^2 -\varepsilon D^2 > 0$ due to the choice of $\varepsilon <
\frac{d^2}{D^2}$. Now, since $ \|T(b)- \sum\limits_{i=1}^n \lambda_i
T(a_i) \|^2>0$ for all choices of $\lambda$, it follows that $T(b)
\notin\mbox{\sf conv}\{T(a_1), \ldots, T(a_n)\}$. In summary, if
$S_\varepsilon$ occurs, then $T(b) \notin \mbox{\sf conv}
\{T(a_1),\ldots, T(a_n)\}$.  Thus, by Lemma
\ref{randomprojectionlemma2} and the union bound,
{\small
\begin{equation*}
\mbox{\sf Prob} (T(b) \notin T(C)) \ge \mbox{\sf Prob} (S_\varepsilon)
\ge 1 - 2\big(n + 2{\scriptsize {n \choose 2}}\big)
e^{-\mathcal{C}(\varepsilon^2-\varepsilon^3) k} = 1 -2
n^2e^{-\mathcal{C}(\varepsilon^2-\varepsilon^3) k}
\end{equation*}
}
for some constant $\mathcal{C} > 0$. \qed
\end{proof}

In order to deal with the CM problem, we consider the $A$-norm of
$x\in \mbox{\sf cone}\{a_1,\ldots,a_n\}$ as $\|x\|_A = \min
\big\{\sum\limits_{i=1}^{n} \lambda_i \;\big|\;\lambda \ge 0\land x =
\sum\limits_{i=1}^{n} \lambda_i a_i \big\}$. For each $x \in \mbox{\sf
  cone}\{a_1,\ldots,a_n\}$, we say that $\lambda\in\mathbb{R}^n_+$
yields a \emph{minimal $A$-representation} of $x$ if and only if
$\sum\limits_{i=1}^{n} \lambda_i = \|x\|_A$. We define $\mu_A = \max
\{\|x\|_A \;|\; x \in \mbox{\sf cone}\{a_1,\ldots,a_n\} \land \|x\|
\le 1\}$; then, for all $x \in \mbox{\sf cone}\{a_1,\ldots,a_n\}$,
$\|x\| \le \|x\|_A \le \mu_A \|x\|$. In particular $\mu_A \ge 1$. Note
that $\mu_A$ serves as a measure of worst-case distortion when we move
from Euclidean to $\|\cdot\|_A$ norm.

\begin{thm} 
Given $b, a_1,\ldots,a_n \in \mathbb{R}^{m}$ of norms $1$ such that
$b\notin C= \mbox{\sf cone} \{a_1,\ldots,a_n\}$, let $d =
\min\limits_{x \in C} \|b-x\|$ and $T:\mathbb{R}^m \to \mathbb{R}^k$
be one of the random mappings in Lemma
\ref{randomprojectionlemma2}. Then:
\begin{equation}
  \mbox{\sf Prob}(\,T(b) \notin \mbox{\sf cone} \{T(a_1),\ldots,T(a_n)\}\,) \ge 1 - 2n(n+1)e^{-\mathcal{C}(\varepsilon^2-\varepsilon^3)k}
  \label{eqnormA}
\end{equation}
for some constant $\mathcal{C}$ (independent of $m,n,k,d$), in which 
$\varepsilon = \frac{d^2}{ \mu_A^2 + 2 \|p\| \mu_A + 1}$.
\label{mainthm}
\end{thm}
\begin{proof}
For any $0 < \varepsilon < 1$, let $S_\varepsilon$ be the event that
both $(1-\varepsilon)\|x-y\|^2 \le \|T(x-y)\|^2 \le
(1+\varepsilon)\|x-y\|^2$ and $(1-\varepsilon)\|x+y\|^2 \le
\|T(x+y)\|^2 \le (1+\varepsilon)\|x+y\|^2$ hold for all $x,y \in
\{b,a_1,\ldots,a_n\}$. By Lemma \ref{randomprojectionlemma2}, we have
$$\mbox{\sf Prob} (S_\varepsilon) \ge 1 - 4{n+1 \choose
  2}e^{-\mathcal{C}(\varepsilon^2-\varepsilon^3) k} = 1 -
2n(n+1)e^{-\mathcal{C}(\varepsilon^2-\varepsilon^3) k}$$ for some
constant $\mathcal{C}$ (independent of $m,n,k,d$). We will prove that
if $S_\varepsilon$ occurs, then we have $T(b) \notin \mbox{\sf cone}
\{T(a_1),\ldots, T(a_n)\}$. Assume that $S_\varepsilon$
occurs. Consider an arbitrary $x \in \mbox{\sf
  cone}\{a_1,\ldots,a_n\}$ and let $\sum\limits_{i=1}^{n} \lambda_i
a_i$ be the minimal $A$-representation of $x$. Then we have:
{\small
\begin{eqnarray*}
&&\|T(b) - T(x)\|^2 =  \|T(b) - \sum_{i=1}^n \lambda_i T(a_i) \|^2 \\ [-0.2em]
&= & \|T(b)\|^2 + \sum_{i=1}^n \lambda_i^2 \|T(a_i)\|^2 - 2\sum_{i=1}^n \lambda_i \langle T(b) ,T(a_i) \rangle +   2\sum_{1 \le i < j \le n} \lambda_i\lambda_j \langle T(a_i) ,T(a_j) \rangle \\ [-0.2em]
&=&\!\!\|T(b)\|^2\!\!+\!\!\sum_{i=1}^n\!\lambda_i^2 \|T(a_i)\|^2\!\!+\!\!\sum_{i=1}^n \frac{\lambda_i}{2} (\|T(b\!-\!a_i)\|^2\!\!-\!\|T(b\!+\!a_i)\|^2 )\!+\!\!\!\!\!\!\sum_{1 \le i < j \le n}\!\!\!\!\!\!\frac{\lambda_i\lambda_j}{2}(\|T(a_i\!+\!a_j)\|^2\!\!-\!\|T(a_i\!-\!a_j)\|^2) \\ [-0.2em]
&\ge & (1 - \varepsilon) \|b\|^2 + (1 - \varepsilon) \sum_{i=1}^n \lambda_i^2 \|a_i\|^2 + \sum_{i=1}^n \frac{\lambda_i}{2} ((1 - \varepsilon)\|b - a_i\|^2 - (1 + \varepsilon) \|b+a_i\|^2 ) \\ && \hspace*{4cm} + \sum_{1 \le i < j \le n}\frac{\lambda_i\lambda_j}{2} ((1 - \varepsilon) \|a_i + a_j\|^2 - (1 + \varepsilon) \|a_i-a_j\|^2 ),
\end{eqnarray*}
}%
because of the assumption that $S_\varepsilon$ occurs. Since
$\|b\| = \|a_1\| = \ldots \|a_n\| = 1$, the RHS can be written as
\begin{eqnarray*}
&&  \|b - \sum_{i=1}^n \lambda_i a_i \|^2 
 - \varepsilon \bigg (1 + \sum_{i=1}^n \lambda_i^2 + 2 \sum_{i=1}^n \lambda_i + 2 \sum_{j \neq i} \lambda_i \lambda_j \bigg) \\
&= & \|b - \sum_{i=1}^n \lambda_i a_i \|^2 
- \varepsilon \big (1 + \sum_{i=1}^n \lambda_i \big)^2  \\
&= & \|b - x\|^2 
- \varepsilon \big (1 + \|x\|_A \big)^2  
\end{eqnarray*}
Denote by $\alpha=\|x\|$ and let $p$ be the projection of $b$ to
$\mbox{\sf cone}\{a_1,\ldots,a_n\}$, which implies $\|b-p\|=\min
\{\|b-x\|\;|\;x \in \mbox{\sf cone}\{a_1,\ldots,a_n\}\}$.
\begin{quote}
  {\bf Claim}. For all $b,x,\alpha,p$ given above, we have
  $\|b-x\|^2\ge\alpha^2-2 \alpha\|p\|+1$.
\end{quote}
By this claim (proved later), we have:
\begin{eqnarray*}
 \|T(b) - T(x) \|^2  & >& \alpha^2 - 2\alpha\|p\| + 1 - \varepsilon
 \big (1 + \|x\|_A \big)^2 \\ 
 \ge \alpha^2 - 2\alpha\|p\| + 1 - \varepsilon \big (1 + \mu_A
 \alpha\big)^2 & = & \big(1 - \varepsilon \mu_A^2\big) \alpha^2 - 2 \big( \|p\| +
 \varepsilon \mu_A\big)\alpha + (1-\varepsilon). 
 \end{eqnarray*}
The last expression can be viewed as a quadratic function with respect
to $\alpha$. We will prove this function is nonnegative for all
$\alpha\in\mathbb{R}$. This is equivalent to
\begin{eqnarray*} 
&&\big(\|p\| + \varepsilon \mu_A\big)^2 -  \big(1 - \varepsilon \mu_A^2 \big) (1-\varepsilon) \le 0 \\
&\Leftrightarrow&  \big(\mu_A^2 +  2 \|p\|\mu_A + 1\big)\varepsilon  \le 1 - \|p\|^2  \\
& \Leftrightarrow & \varepsilon  \le  \frac{1 - \|p\|^2}{\mu_A^2  + 2 \|p\|\mu_A + 1} =  \frac{d^2}{\mu_A^2  + 2 \|p\|\mu_A + 1},
\end{eqnarray*}
which holds for the choice of $\varepsilon$ as in the hypothesis.  In
summary, if the event $S_\varepsilon$ occurs, then $\|T(b) - T(x)\|^2
> 0$ for all $x \in \mbox{\sf cone}\{a_1,\ldots,a_n\}$, i.e.  $T(x)
\notin \mbox{\sf cone}\{T(a_1),\ldots,T(a_n)\}$. Thus,
\[\mbox{\sf Prob}(T(b) \notin TC) \ge \mbox{\sf Prob} (S_\varepsilon) \ge 1 - 2n(n+1)e^{-c(\varepsilon^2 - \varepsilon^3)k}\]
as claimed. \qed
\par\bigskip\par
{
\noindent {\it Proof of the claim}. If $x = 0$ then the claim is
trivially true, since $\|b-x\|^2= \|b
\|^2=1=\alpha^2-2\alpha\|p\|+1$. Hence we assume $x\neq 0$.  First
consider the case $p\neq 0$. By Pythagoras' theorem, we must have $d^2
= 1 - \|p\|^2$.  We denote by $z=\frac{\|p\|}{\alpha}x$, then $\|z\| =
\|p\|$. Set $\delta = \frac{\alpha}{\|p\|}$, we have
\begin{eqnarray*}
\|b - x \|^2  & =& \|b - \delta z \|^2 \\
& =& (1 - \delta) \|b\|^2 + (\delta^2 - \delta) \|z\|^2 + \delta \|b - z\|^2 \\
& =& (1 - \delta) + (\delta^2 - \delta) \|p\|^2 + \delta \|b - z\|^2\\
& \ge& (1 - \delta) + (\delta^2 - \delta) \|p\|^2 + \delta d^2\\
& =& (1 - \delta) + (\delta^2 - \delta) \|p\|^2 + \delta (1 - \|p\|^2) \\
& =& \delta^2 \|p\|^2 - 2\delta \|p\|^2 + 1  = \alpha^2 - 2\alpha\|p\| + 1.
\end{eqnarray*}
Next, we consider the case $p = 0$. In this case we have $b^T(x) \le 0$
for all $x \in \mbox{\sf cone}\{a_1,\ldots,a_n\}$. Indeed, for an arbitrary $\delta >0$,
\[0 \le \frac{1}{\delta} (\|b-\delta x\|^2 - 1) = \frac{1}{\delta} (1 + \delta^2 \|x\|^2 - 2 \delta b^Tx - 1) = \delta \|x\|^2 - 2 b^Tx\]
which tends to $-2 b^Tx$ when $\delta \to 0^+$. Therefore
\begin{eqnarray*}
\|b - x \|^2 & =& 1 - 2b^Tx + \|x\|^2 \ge \|x \|^2 + 1 = \alpha^2 -
2\alpha\|p\| + 1,
\end{eqnarray*}
which proves the claim.} \qed
\end{proof}

\section{Computational results}
Let $T$ be the random projector, $A$ the constraint matrix, $b$ the
RHS vector, and $X$ either $\mathbb{R}^n_+$ in the case of LP and
$\mathbb{Z}^n_+$ in the case of IP. We solve $Ax=b\land x\in X$ and
$T(A)x=T(b)\land x\in X$ to compare accuracy and performance. In these
results, $A$ is dense. We generate $(A,b)$ componentwise from three
distributions: uniform on $[0,1]$, exponential, gamma. For $T$, we
only test the best-known type of projector matrix $T(y)=Py$, namely
$P$ is a random $k\times m$ matrix each component of which is
independently drawn from a normal $\mathcal{N}(0,\frac{1}{\sqrt{k}})$
distribution. All problems were solved using CPLEX 12.6 on an Intel i7
2.70GHz CPU with 16.0 GB RAM. All the computational experiments were
carried out in JuMP \cite{jump}.

Because accuracy is guaranteed for feasible instances by Lemma
\ref{membershiplemma} (i), we only test {\it infeasible} LP and IP
feasibility instances. For every given size $m\times n$ of the
constraint matrix, we generate 10 different instances, each of which
is projected using 100 randomly generated projectors $P$. For each
size, we compute the percentage of success, defined as an infeasible
original problem being reduced to an infeasible projected
problem. Performance is evaluated by recording the average user CPU
time taken by CPLEX to solve the original problem, and, comparatively,
the time taken to perform the matrix multiplication $PA$ plus the time
taken by CPLEX to solve the projected problem. 

\begin{table}
  \begin{center}
    \caption{\small LP: above, IP: below. {\it Acc.}: accuracy (\% feas./infas.~agreement), {\it Orig.}: original (CPU), {\it Proj.}: projected instances (CPU).} 
    \label{tab_results}
 {\small
   \begin{tabular}{ll|rrr|rrr|rrr}
      \hline
      & & \multicolumn{3}{c|}{Uniform} &
      \multicolumn{3}{c|}{Exponential} & \multicolumn{3}{c}{Gamma} \\[-0.2em]
      $m$ & $n$ & Acc. & Orig. & Proj. & Acc. & Orig. & Proj. & Acc. & Orig. & Proj. \\
      \hline
      600  & 1000 & $99.5\%$ & 1.57s & 0.12s & $93.7\%$ & 1.66s & 0.12s & $94.6\%$ & 1.64s & 0.11s \\ [-0.2em]
      700  & 1000 & $99.5\%$ & 2.39s & 0.12s & $92.8\%$ & 2.19s & 0.12s & $93.1\%$ & 2.15s & 0.11s \\ [-0.2em]
      800  & 1000 & $99.5\%$ & 2.55s & 0.11s & $95.0\%$ & 2.91s& 0.11s & $97.3\%$ & 2.78 s& 0.11s\\ [-0.2em]
      900  & 1000 & $99.6\%$ & 3.49s & 0.12s& $96.1\%$ & 3.65s& 0.13s & $97.0\%$ & 3.57s& 0.13s\\ [-0.2em]
      1000 & 1500 & $99.5\%$ & 16.54s & 0.20s & $93.0\%$ & 18.10s & 0.20s & $91.2\%$ & 17.58s & 0.20s\\ [-0.2em]
      1200 & 1500 & $99.6\%$ & 22.46s& 0.23s & $95.7\%$ & 22.46s & 0.20s& $95.7\%$ & 22.58s & 0.22s\\ [-0.2em]
      1400 & 1500 & $100 \%$ & 31.08s& 0.24s& $93.2\%$ & 35.24s& 0.26s& $95.0\%$ & 31.06s& 0.23s\\ [-0.2em]
      1500 & 2000 & $99.4\%$ & 48.89s & 0.30s& $91.3\%$ & 51.23s & 0.31s & $90.1\%$ & 51.08s& 0.31\\ [-0.2em]
      1600 & 2000 & $99.8\%$ & 58.36s& 0.34s& $93.8\%$ & 58.87s & 0.34s& $95.9\%$ & 60.35s& 0.358s\\ 
      \hline
      500 &  800 & $100\%$ & 20.32s& 4.15s& $100\%$ & 4.69s& 10.56s& $100\%$ & 4.25s& 8.11s \\ [-0.2em]
      600 &  800 & $100\%$ & 26.41s& 4.22s& $100\%$ & 6.08s& 10.45s& $100\%$ & 5.96s& 8.27s\\ [-0.2em]
      700 &  800 & $100\%$ & 38.68s & 4.19s& $100\%$ & 8.25s& 10.67s& $100\%$ & 7.93s &  10.28s\\ [-0.2em]
      600 & 1000 & $100\%$ & 51.20s& 7.84s& $100\%$ & 10.31s& 8.47s& $100\%$ & 8.78s& 6.90s\\ [-0.2em]
      700 & 1000 & $100\%$ & 73.73s& 7.86s& $100\%$ & 12.56s& 10.91s& $100\%$ & 9.29s & 8.43s \\ [-0.2em]
      800 & 1000 & $100\%$ & 117.8s& 8.74s& $100\%$ & 14.11s& 10.71s& $100\%$ & 12.29s & 7.58s \\ [-0.2em]
      900 & 1000 & $100\%$ & 130.1s& 7.50s& $100\%$ & 15.58s& 10.75s& $100\%$ & 15.06s& 7.65s\\ [-0.2em]
      1000& 1500 & $100\%$ & 275.8s& 8.84s& $100\%$ & 38.57s& 22.62s& $100\%$ & 35.70s & 8.74s\\ 
      \hline                                              
    \end{tabular}
}
  \end{center}
\end{table}
In the above computational results, we only report the actual solver
execution time (in the case of the original problem) and matrix
multiplication plus solver execution (in the case of the projected
problem).  Lastly, although Tables \ref{tab_results} tell a successful
story, we obtained less satisfactory results with other distributions.
Sparse instances yielded accurate but poorly performant results. So
far, this seems to be a good practical method for dense LP/IP.

\bibliographystyle{plain}
\bibliography{dr2}

\newpage
\end{document}